\documentclass[11pt,a4paper]{article}
\usepackage[margin=30mm]{geometry}
\usepackage[T1]{fontenc}
\usepackage{lmodern}
\usepackage{microtype}
\usepackage{amsmath,amssymb,amsthm,mathtools,mathrsfs}
\usepackage{enumitem}
\usepackage{booktabs}
\usepackage{array}
\usepackage[hidelinks]{hyperref}
\newtheorem{theorem}{Theorem}[section]
\newtheorem{lemma}[theorem]{Lemma}
\newtheorem{proposition}[theorem]{Proposition}
\newtheorem{corollary}[theorem]{Corollary}
\theoremstyle{definition}
\newtheorem{definition}[theorem]{Definition}
\newtheorem{example}[theorem]{Example}
\theoremstyle{remark}
\newtheorem{remark}[theorem]{Remark}

\title{\Large Gradient Ricci Solitons on PTR-Submanifolds\\
of Complex Space Forms}
\author{Jae Won Lee\\
Department of Mathematics Education and RINS, Gyeongsang National University,\\
501 Jinju-daero, Jinju-si, Gyeongsangnam-do 52828, Republic of Korea\\
\texttt{leejw@gnu.ac.kr}}
\date{}

\begin{document}
\maketitle
\vspace{-1.5em}

\begin{abstract}
We study gradient Ricci solitons on partially totally real (PTR) submanifolds
of complex space forms.  In nonzero ambient holomorphic sectional curvature,
the Gauss--soliton equation reconstructs the square $P^2$ of the tangential
K\"ahler operator from the Hessian endomorphism and an extrinsic Ricci
operator.  Consequently, the PTR tangency spectrum is encoded by intrinsic
soliton data together with the second fundamental form.  For compact minimal
PTR solitons with positive ambient curvature we obtain
$\lambda\le c(n-1)/4$, with a strict version under a relative-nullity
condition.  We also prove compact triviality criteria, an exact trace-free
anisotropy identity, a parallel-data rigidity result, and a nonexistence
theorem for compact real hypersurfaces.  A flat multi-level Gaussian model and
a nonflat Segre $CR$-product provide explicit consistency checks and identify
the limiting cases of the reconstruction formulas.
\end{abstract}



\section{Introduction}

Gradient Ricci solitons connect curvature to the Hessian of a potential
function through
\[
\operatorname{Ric}+\operatorname{Hess}f=\lambda g.
\]
For submanifolds of K\"ahler manifolds, this equation naturally interacts
with the second fundamental form and with the tangential component of the
ambient complex structure.  A basic question is whether the soliton equation
merely constrains the induced metric or can also recover genuinely
submanifold-theoretic information.  The purpose of this paper is to show that,
for a broad class of partially totally real submanifolds in a nonflat complex
space form, the latter phenomenon does occur: the soliton and extrinsic data
recover the pointwise complex-tangency operator.

Partially totally real (PTR) submanifolds were introduced by Poyraz,
\c{S}ahin and Yerlikaya as a framework containing totally real, proper $CR$,
hemi-slant and pointwise hemi-slant submanifolds
\cite{PoyrazSahinYerlikaya2025}.  Related partially slant geometry was
developed in \cite{YerlikayaPoyrazSahin2025}.  The PTR tangent bundle splits
into a totally real distribution and an orthogonal complementary distribution
preserved by the tangential component $P$ of the almost complex structure.
Because $P$ is skew-adjoint, $P^2$ is self-adjoint and its eigenvalues are the
negatives of the squared pointwise tangency levels.  Thus recovering $P^2$
amounts to recovering the PTR tangency spectrum.

For general background on gradient Ricci solitons and rigidity, see Cao
\cite{Cao2010} and Petersen--Wylie \cite{PetersenWylie2009b}.  In K\"ahler
submanifold geometry, Cho--Kimura studied Ricci solitons on Lagrangian
submanifolds \cite{ChoKimura2010} and on real hypersurfaces of complex space
forms \cite{ChoKimura2009,ChoKimura2011}; see also Chen
\cite{Chen2014} for a broader submanifold perspective.  Cidre-D\'iaz and
Sanmart\'in-L\'opez recently classified homogeneous expanding Ricci-soliton
Lie subgroups of complex hyperbolic spaces
\cite{CidreDiazSanmartinLopez2026}.  Recent developments adjacent to the PTR
setting include warped-product PTR geometry \cite{YerlikayaSahin2026} and
Chen-type inequalities for partially slant submanifolds \cite{Aquib2026}.
These works motivate, but do not contain, the reconstruction problem studied
here.

The central identity of the paper expresses the Hessian endomorphism $S_f$,
the self-adjoint extrinsic Ricci contribution $\mathcal Q_h$, and $P^2$ in a
single operator equation.  When the ambient holomorphic sectional curvature
$c$ is nonzero, it can be solved for $P^2$.  This gives the bridge
\[
(S_f,\mathcal Q_h)\quad\Longrightarrow\quad P^2
\quad\Longrightarrow\quad
\text{PTR tangency spectrum}.
\]
The significance of the identity is therefore not the algebraic rearrangement
itself, but the rigidity information that follows from it.  On the totally
real block it determines a distinguished eigenspace.  In the compact minimal
case with $c>0$, a maximum-principle argument yields
$\lambda\le c(n-1)/4$ and a relative-nullity hypothesis makes the inequality
strict.  A scalar contraction gives a compact triviality criterion.  On the
ambiguous block, the reconstructed operator yields spectral separation,
endpoint detection, and an exact trace-free anisotropy formula.  After
differentiation it also gives a curvature--Hessian compatibility condition;
parallel data then force further rigidity.  Finally, the compact criterion
produces a nonexistence result for real hypersurfaces of nonflat complex space
forms.

The examples are used as checks rather than as substitutes for the general
arguments.  A flat affine PTR Gaussian soliton exhibits several nonzero
pointwise tangency levels but also shows why nonzero ambient curvature is
needed for reconstruction.  A standard nonflat Segre $CR$-product realizes
the relevant spectral endpoints and verifies the trace-free formulas, while
its constant potential illustrates the compact triviality theorem.  At
present no nonflat PTR example with nonconstant potential and several
distinct nonzero tangency levels is known to us, and this limitation is stated
explicitly in the final section.

Section~2 gives the PTR set-up and the benchmark examples.  Section~3
develops the scalar, maximum-principle, and compact rigidity consequences.
Section~4 contains the operator reconstruction, spectral results, and the
differential compatibility condition.  Section~5 records the real
hypersurface consequence, comparisons, and the remaining existence question.

\section{PTR geometry and the soliton set-up}

Let $(\widetilde M,J,g)$ be a K\"ahler manifold and $N$ an $n$-dimensional Riemannian submanifold. For $X\in TN$, write
\begin{equation}\label{eq:Jdecomp}
JX=PX+FX,
\end{equation}
where $PX$ and $FX$ are the tangential and normal components, respectively. If $\xi$ is normal, we similarly write
\[
J\xi=t\xi+\mathfrak f\xi,
\]
where $t\xi$ and $\mathfrak f\xi$ are the tangential and normal components. The original PTR notation uses $f$ for the normal component; here we write $\mathfrak f$ in order to reserve $f:N\to\mathbb R$ for the gradient Ricci soliton potential. Since $J$ is compatible with $g$,
\begin{equation}\label{eq:Pskew}
g(PX,Y)=-g(X,PY).
\end{equation}
Thus $P$ is skew-adjoint.

\begin{definition}[PTR-submanifold \cite{PoyrazSahinYerlikaya2025}]\label{def:PTR}
A submanifold $N$ is called \emph{partially totally real} if there is a totally real distribution $\mathscr H^{\perp}$ such that
\[
TN=\mathscr H^{\perp}\oplus\widetilde{\mathscr H},
\qquad
\mathscr H^{\perp}\perp\widetilde{\mathscr H},
\qquad
P(\widetilde{\mathscr H})\subseteq\widetilde{\mathscr H}.
\]
The orthogonal complementary distribution $\widetilde{\mathscr H}$ is called the ambiguous distribution.
\end{definition}

\noindent\textbf{Convention on the limiting cases.}
We follow the convention stated explicitly in Remark~1 of
Poyraz--\c{S}ahin--Yerlikaya \cite{PoyrazSahinYerlikaya2025}: totally real
submanifolds are included among PTR-submanifolds, whereas holomorphic
submanifolds are not. Accordingly, the defining totally real distribution
$\mathscr H^\perp$ is taken to have positive rank in the PTR situation used
here; the totally real limiting case $\widetilde{\mathscr H}=0$ is allowed.
Holomorphic submanifolds will therefore appear below only as external
K\"ahler comparisons, not as PTR examples.

\begin{definition}[Proper PTR-submanifold \cite{PoyrazSahinYerlikaya2025}]
A PTR-submanifold is called \emph{proper} if the ambiguous distribution
$\widetilde{\mathscr H}$ is neither totally real nor complex, slant, or
pointwise slant.
\end{definition}

\begin{remark}[A maximal choice on constant-rank regions]\label{rem:canonicalPTR}
The bare PTR condition is deliberately broad.  On an open set where
\(\operatorname{rank}P\) is constant and \(\ker P\) has positive rank,
skew-adjointness of \(P\) gives
\[
TN=\ker P\oplus(\ker P)^\perp
=\ker P\oplus\operatorname{Im}P.
\]
If \(X\in\ker P\), then the tangential component of \(JX\) vanishes, hence
\(JX\) is normal; thus \(\ker P\) is totally real.  Moreover,
\[
P(\operatorname{Im}P)\subseteq\operatorname{Im}P.
\]
Consequently, on every constant-rank region for which \(\ker P\) has
positive rank, the maximal choice
\(\mathscr H^\perp=\ker P\) and
\(\widetilde{\mathscr H}=\operatorname{Im}P\) automatically satisfies the
PTR conditions of Definition~\ref{def:PTR}.  Accordingly, the geometric
content considered here concerns the restrictions imposed by the gradient
Ricci soliton equation on this maximal splitting and on the nonzero spectrum
of \(P^2\); in the proper PTR case, this includes the nontrivial spectral
structure of the ambiguous distribution.

This maximal choice is an additional standing convention on the constant-rank
region under consideration; it is not part of the PTR definition itself.  To
remove the nonuniqueness coming from smaller totally real subdistributions,
throughout the spectral discussion we use
\[
\mathscr H^\perp=\ker P,
\qquad
\widetilde{\mathscr H}=\operatorname{Im}P=(\ker P)^\perp.
\]
Hence
\[
q=\dim\widetilde{\mathscr H}=\operatorname{rank}P
\]
is intrinsic to the immersion on each constant-rank component.  Since the
rank of a real skew-adjoint endomorphism is even, this maximal ambiguous
dimension \(q\) is automatically even.  The geometric content studied below
is therefore not the mere existence of a PTR splitting, but the restrictions
imposed by the gradient soliton equation on the kernel and on the nonzero
spectrum of \(P^2\).
\end{remark}

\begin{remark}[Basic special cases]
If $P=0$ on all of $TN$, then $N$ is totally real. If $\widetilde{\mathscr H}$ is $J$-invariant, then $P^2=-I$ on $\widetilde{\mathscr H}$, and the structure reduces to the CR case. If
\[
P^2=-\rho^2I\quad\text{on }\widetilde{\mathscr H}
\]
for a pointwise function $0<\rho<1$, then the ambiguous distribution has one pointwise slant level; in the terminology already contained in the PTR framework, this is the pointwise hemi-slant special situation. Pointwise slant immersions in complex space forms have an independent existence and uniqueness theory; see, for example, \cite{AlghanemiEtAl2021}.
\end{remark}

\begin{lemma}[Elementary spectral properties of $P^2$]\label{lem:P2basic}
On $TN$, the operator $P^2$ is self-adjoint and nonpositive. Moreover, every eigenvalue of $-P^2$ lies in $[0,1]$. Equivalently, if
\[
P^2X=-\rho^2X,
\qquad X\ne0,
\]
then $0\le \rho^2\le1$.  The rank of \(P\) is even; more precisely, every
positive eigenvalue of \(-P^2\) has even multiplicity.
\end{lemma}

\begin{proof}
We verify the assertions one at a time.

From
\[
g(PX,Y)=-g(X,PY)
\]
for all tangent vectors \(X,Y\), the adjoint of \(P\) satisfies
\[
P^*=-P.
\]
Hence
\[
(P^2)^*=(P^*)^2=(-P)^2=P^2,
\]
so \(P^2\) is self-adjoint.

Next, put \(Y=PX\) in the skew-adjointness relation:
\[
g(PX,PX)=-g(X,P(PX))=-g(X,P^2X).
\]
Therefore
\[
g(P^2X,X)=-\|PX\|^2\le0,
\]
which proves that \(P^2\) is nonpositive.

Now suppose
\[
P^2X=-\rho^2X,\qquad X\ne0.
\]
Taking the inner product with \(X\),
\[
g(P^2X,X)=-\rho^2\|X\|^2.
\]
Using the preceding identity,
\[
-\|PX\|^2=-\rho^2\|X\|^2,
\]
and hence
\[
\|PX\|^2=\rho^2\|X\|^2.
\]
Since
\[
JX=PX+FX
\]
is the orthogonal decomposition of \(JX\) into tangent and normal parts,
\[
\|JX\|^2=\|PX\|^2+\|FX\|^2\ge \|PX\|^2.
\]
Because \(J\) is an isometry,
\[
\|JX\|=\|X\|.
\]
Thus
\[
\rho^2\|X\|^2=\|PX\|^2\le\|X\|^2.
\]
Dividing by \(\|X\|^2>0\) gives
\[
0\le\rho^2\le1.
\]

It remains to prove the parity statement.  Let \(E_{\rho}\) be the eigenspace
of \(-P^2\) corresponding to \(\rho^2>0\).  Since \(P\) commutes with \(P^2\),
the space \(E_{\rho}\) is \(P\)-invariant.  On \(E_{\rho}\), define
\[
J_{\rho}=\frac1{\rho}P.
\]
Then
\[
J_{\rho}^2=\frac1{\rho^2}P^2=-I.
\]
Thus \(E_{\rho}\) is a complex vector space with respect to \(J_{\rho}\), so
\(\dim_{\mathbb R}E_{\rho}\) is even. Summing over the nonzero spectral
blocks shows that \(\operatorname{rank}P\) is even.
\end{proof}

\begin{remark}
At points where the multiplicities of the eigenvalues of $P^2|_{\widetilde{\mathscr H}}$ are locally constant, the ambiguous distribution admits the corresponding orthogonal local spectral decomposition. We use this only pointwise below; no integrability assumption on the eigendistributions is required.
\end{remark}

Let $h$ be the second fundamental form and
\[
H=\frac1n\operatorname{trace}h
\]
the mean-curvature vector. Define the self-adjoint operator $\mathcal Q_h$ by
\begin{equation}\label{eq:Qhdef}
g(\mathcal Q_hX,Y)
=
n\,g(H,h(X,Y))
-\sum_{i=1}^{n}g(h(X,e_i),h(Y,e_i)),
\end{equation}
where $\{e_i\}_{i=1}^n$ is a local orthonormal tangent frame. This notation merely packages the two second-fundamental-form terms appearing in the Ricci contraction.

\begin{remark}[Two immediate cases]
If $N$ is totally geodesic, then $h=0$ and $\mathcal Q_h=0$. If $N$ is minimal, then $H=0$ and
\[
g(\mathcal Q_hX,X)
=-\sum_i\|h(X,e_i)\|^2\le0.
\]
Thus in the minimal case $\mathcal Q_h$ is nonpositive in the quadratic-form sense.
\end{remark}

We now give an actual PTR gradient soliton before entering the nonflat theory.

\begin{example}[A flat proper PTR Gaussian soliton]\label{ex:flatPTR}
Work in $\mathbb C^4\cong\mathbb R^8$ with an orthonormal real frame
\[
\{e_1,Je_1,e_2,Je_2,e_3,Je_3,e_4,Je_4\}.
\]
Fix $0<\theta<\pi/2$, and define
\[
u=e_2,
\qquad
v=\cos\theta\,Je_2+\sin\theta\,e_3.
\]
Let
\[
\widetilde{\mathscr H}
=\operatorname{span}\{e_1,Je_1,u,v\},
\qquad
\mathscr H^\perp=\operatorname{span}\{e_4\},
\]
choose $x_0\in\mathbb C^4$, and let
$N=x_0+(\widetilde{\mathscr H}\oplus\mathscr H^\perp)$ be the affine
linear subspace with this direction space.
Then $J\mathscr H^\perp$ is normal. A direct projection gives
\[
Pu=\cos\theta\,v,
\qquad
Pv=-\cos\theta\,u.
\]
Hence
\[
P^2=-I
\quad\text{on }\operatorname{span}\{e_1,Je_1\},
\qquad
P^2=-\cos^2\theta\,I
\quad\text{on }\operatorname{span}\{u,v\}.
\]
Thus the ambiguous distribution has two distinct nonzero tangency levels; in particular this is a genuine multi-level proper PTR model.

Because $N$ is affine linear, $h=0$ and $\operatorname{Ric}=0$. Fix
$\lambda\in\mathbb R\setminus\{0\}$ and define
\[
f(x)=\frac{\lambda}{2}\|x-x_0\|^2,\qquad x\in N,
\]
where $\|\cdot\|$ is the Euclidean norm induced on the affine subspace.
Then $\nabla f=\lambda(x-x_0)$ under the canonical identification with the
direction space and $\operatorname{Hess}f=\lambda g$. Therefore
\[
\operatorname{Ric}+\operatorname{Hess}f=\lambda g.
\]
This gives a proper PTR gradient Ricci soliton with nonconstant potential in the flat case. Notice that the example also explains why the nonflat assumption will matter later: when $c=0$, the soliton equation does not directly detect the different $P^2$-levels.
\end{example}

\begin{example}[A nonflat PTR Einstein soliton from a standard Segre \(CR\)-product]
\label{ex:SegrePTR}
Consider the standard Segre \(CR\)-product
\[
\Phi:
\mathbb{CP}^{1}(4)\times\mathbb{RP}^{5}(1)
\longrightarrow
\mathbb{CP}^{11}(4),
\]
obtained by composing
\[
\mathbb{CP}^{1}(4)\times\mathbb{RP}^{5}(1)
\hookrightarrow
\mathbb{CP}^{1}(4)\times\mathbb{CP}^{5}(4)
\]
with the Segre embedding \(S_{1,5}\). Standard \(CR\)-products of this form
were studied by Chen \cite{Chen1981b}.

Let
\[
\mathscr D=T\mathbb{CP}^{1},
\qquad
\mathscr D^\perp=T\mathbb{RP}^{5}.
\]
Then \(\mathscr D\) is \(J\)-invariant and \(\mathscr D^\perp\) is totally
real.  With the canonical maximal decomposition of
Remark~\ref{rem:canonicalPTR},
\[
\widetilde{\mathscr H}=\mathscr D,
\qquad
\mathscr H^\perp=\mathscr D^\perp,
\]
and hence
\[
P^2=-I\quad\text{on }\mathscr D,
\qquad
P=0\quad\text{on }\mathscr D^\perp.
\]
Since \(\dim_{\mathbb R}\mathscr D=2\),
\[
n=7,\qquad \|P\|^2=2.
\]

The induced metric is the product metric.  The factors satisfy
\[
\operatorname{Ric}_{\mathbb{CP}^{1}(4)}=4g,
\qquad
\operatorname{Ric}_{\mathbb{RP}^{5}(1)}=4g,
\]
so
\[
\operatorname{Ric}_N=4g.
\]
Taking \(f\) constant gives
\[
\operatorname{Hess}f=0,
\qquad
\operatorname{Ric}+\operatorname{Hess}f=4g.
\]
Thus this is a nonflat, non-totally-geodesic PTR gradient Ricci soliton (in the
$CR$ special case, hence not a proper PTR example) with
\(\lambda=4\), but the soliton is \emph{trivial} in the standard sense because
\(f\) is constant.  It is therefore a benchmark for the nonflat identities,
not an example of a nontrivial nonflat soliton.
\end{example}

We reserve $R$ and $\widetilde R$ for the curvature tensors of $N$ and
$\widetilde M$, respectively, and write $\operatorname{Scal}_N$ for the
scalar curvature of $N$.

Assume henceforth that the ambient manifold is a complex space form $\widetilde M(c)$ of constant holomorphic sectional curvature $c$, with curvature tensor
\begin{align}\label{eq:ambientcurv}
\widetilde R(U,V)W
=\frac c4\{&g(V,W)U-g(U,W)V
+g(JV,W)JU \notag\\
&-g(JU,W)JV+2g(U,JV)JW\}.
\end{align}
We use the Gauss equation in the form
\begin{align}\label{eq:Gauss}
g(R(X,Y)Z,W)
={}&g(\widetilde R(X,Y)Z,W) \notag\\
&+g(h(X,W),h(Y,Z))-g(h(X,Z),h(Y,W)).
\end{align}

Finally, a gradient Ricci soliton is a quadruple $(N,g,f,\lambda)$, where
$f\in C^\infty(N)$ and $\lambda\in\mathbb R$ is a constant, satisfying
\begin{equation}\label{eq:soliton}
\operatorname{Ric}+\operatorname{Hess}f=\lambda g.
\end{equation}
Thus we use the standard gradient Ricci soliton equation, not the almost Ricci
soliton equation in which the coefficient of $g$ may vary.
We denote by $S_f$ the Hessian endomorphism,
\[
g(S_fX,Y)=\operatorname{Hess}f(X,Y).
\]

The following contraction is the standard Gauss contraction for a real
submanifold of a complex space form; compare, for example,
Matsumoto--Mihai--Tazawa \cite{MatsumotoMihaiTazawa2003}.  We include the
calculation because the sign and the coefficient of \(P^2\) are used
throughout the paper.

The following contraction will be used in both the scalar and operator arguments.

\begin{lemma}[Ricci contraction]\label{lem:Ricci}
For tangent vector fields $X,Y$ on a PTR-submanifold of $\widetilde M(c)$,
\begin{equation}\label{eq:Ricci}
\operatorname{Ric}(X,Y)
=
\frac c4\bigl\{(n-1)g(X,Y)-3g(P^2X,Y)\bigr\}
+g(\mathcal Q_hX,Y).
\end{equation}
\end{lemma}

\begin{proof}
Let \(\{e_1,\ldots,e_n\}\) be a local orthonormal tangent frame. With our convention,
\[
\operatorname{Ric}(X,Y)
=
\sum_{i=1}^{n}g(R(e_i,X)Y,e_i).
\]
Apply the Gauss equation to each summand:
\begin{align*}
\operatorname{Ric}(X,Y)
={}&
\sum_{i=1}^{n}g(\widetilde R(e_i,X)Y,e_i)\\
&+\sum_{i=1}^{n}g(h(e_i,e_i),h(X,Y))\\
&-\sum_{i=1}^{n}g(h(e_i,Y),h(X,e_i)).
\end{align*}
Since
\[
\sum_{i=1}^{n}h(e_i,e_i)=nH
\]
and \(h\) is symmetric, the last two lines equal
\[
n\,g(H,h(X,Y))
-\sum_{i=1}^{n}g(h(X,e_i),h(Y,e_i))
=
g(\mathcal Q_hX,Y).
\]

It remains to contract the ambient curvature. Substitute
\[
U=e_i,\qquad V=X,\qquad W=Y
\]
into \eqref{eq:ambientcurv}:
\begin{align*}
\widetilde R(e_i,X)Y
=\frac c4\{&
g(X,Y)e_i-g(e_i,Y)X
+g(JX,Y)Je_i\\
&-g(Je_i,Y)JX
+2g(e_i,JX)JY\}.
\end{align*}
Taking the inner product with \(e_i\) and summing, the first term gives
\[
n\,g(X,Y),
\]
while the second gives
\[
-\sum_i g(e_i,Y)g(X,e_i)=-g(X,Y).
\]
Thus these two terms contribute
\[
(n-1)g(X,Y).
\]
The third term vanishes because
\[
g(Je_i,e_i)=0.
\]

For the fourth term,
\begin{align*}
-\sum_i g(Je_i,Y)g(JX,e_i)
&=\sum_i g(e_i,JY)g(JX,e_i)\\
&=\sum_i g(e_i,PY)g(PX,e_i)\\
&=g(PX,PY).
\end{align*}
For the fifth term,
\begin{align*}
2\sum_i g(e_i,JX)g(JY,e_i)
&=2\sum_i g(e_i,PX)g(PY,e_i)\\
&=2g(PX,PY).
\end{align*}
Hence the Kähler contribution is
\[
3g(PX,PY).
\]
Using \(P^*=-P\),
\[
g(PX,PY)
=-g(X,P^2Y)
=-g(P^2X,Y),
\]
because \(P^2\) is self-adjoint. Therefore
\[
\sum_i g(\widetilde R(e_i,X)Y,e_i)
=
\frac c4\bigl\{(n-1)g(X,Y)-3g(P^2X,Y)\bigr\}.
\]
Adding the extrinsic term \(g(\mathcal Q_hX,Y)\) proves \eqref{eq:Ricci}.
\end{proof}

\section{Scalar soliton identities and global consequences}\label{sec:scalar}

Before using the full operator information, it is useful to record what survives after taking the trace. This follows the standard soliton--submanifold practice of combining the soliton equation with the Gauss equation and then separating intrinsic and extrinsic terms; compare the Lagrangian Ricci-soliton setting of Cho--Kimura \cite{ChoKimura2010}. The PTR case has an additional term, namely the total tangency energy $\|P\|^2$.

\begin{proposition}[Scalar PTR soliton identity]\label{prop:scalar}
Let $N^n$ be a PTR-submanifold of a complex space form $\widetilde M(c)$ carrying a gradient Ricci soliton. Then
\begin{equation}\label{eq:scalarPTR}
\Delta f+n^2\|H\|^2-\|h\|^2
=
n\lambda-\frac{cn(n-1)}4-\frac{3c}{4}\|P\|^2.
\end{equation}
Here $\|P\|^2=\sum_i\|Pe_i\|^2$ and $\|h\|^2=\sum_{i,j}\|h(e_i,e_j)\|^2$.
\end{proposition}

\begin{proof}
Let \(\{e_1,\ldots,e_n\}\) be a local orthonormal tangent frame. By definition,
\[
\operatorname{Scal}_N=\sum_{j=1}^{n}\operatorname{Ric}(e_j,e_j).
\]
Using Lemma~\ref{lem:Ricci},
\begin{align*}
\operatorname{Scal}_N
&=\sum_{j=1}^{n}
\left[
\frac c4\{(n-1)g(e_j,e_j)-3g(P^2e_j,e_j)\}
+g(\mathcal Q_he_j,e_j)
\right]\\
&=\frac c4
\left[
n(n-1)-3\sum_{j=1}^{n}g(P^2e_j,e_j)
\right]
+\sum_{j=1}^{n}g(\mathcal Q_he_j,e_j).
\end{align*}
Hence
\[
\operatorname{Scal}_N=\frac c4\{n(n-1)-3\operatorname{tr}P^2\}
+\operatorname{tr}\mathcal Q_h.
\]

We compute the two traces separately. Since \(P\) is skew-adjoint,
\[
g(P^2e_j,e_j)=-g(Pe_j,Pe_j)=-\|Pe_j\|^2.
\]
Therefore
\[
\operatorname{tr}P^2
=\sum_{j=1}^{n}g(P^2e_j,e_j)
=-\sum_{j=1}^{n}\|Pe_j\|^2
=-\|P\|^2.
\]

From the definition of \(\mathcal Q_h\),
\begin{align*}
\operatorname{tr}\mathcal Q_h
&=\sum_{j=1}^{n}g(\mathcal Q_he_j,e_j)\\
&=\sum_{j=1}^{n}\left[
n\,g(H,h(e_j,e_j))
-\sum_{i=1}^{n}\|h(e_j,e_i)\|^2
\right]\\
&=n\,g\left(H,\sum_{j=1}^{n}h(e_j,e_j)\right)
-\sum_{i,j=1}^{n}\|h(e_i,e_j)\|^2.
\end{align*}
Since
\[
H=\frac1n\sum_{j=1}^{n}h(e_j,e_j),
\]
we have
\[
\sum_{j=1}^{n}h(e_j,e_j)=nH.
\]
Thus
\[
\operatorname{tr}\mathcal Q_h
=n\,g(H,nH)-\|h\|^2
=n^2\|H\|^2-\|h\|^2.
\]
Consequently,
\[
\operatorname{Scal}_N=
\frac{cn(n-1)}4
+\frac{3c}{4}\|P\|^2
+n^2\|H\|^2-\|h\|^2.
\]

Finally, trace the soliton equation
\[
\operatorname{Ric}+\operatorname{Hess}f=\lambda g.
\]
The three traces are
\[
\operatorname{tr}(\operatorname{Ric})=\operatorname{Scal}_N,\qquad
\operatorname{tr}(\operatorname{Hess}f)=\Delta f,\qquad
\operatorname{tr}(\lambda g)=n\lambda.
\]
Hence
\[
\operatorname{Scal}_N+\Delta f=n\lambda.
\]
Substitution of the expression for \(\operatorname{Scal}_N\) gives
\[
\Delta f+n^2\|H\|^2-\|h\|^2
=
n\lambda-\frac{cn(n-1)}4-\frac{3c}{4}\|P\|^2.
\]
\end{proof}

\begin{example}[Verification on the flat Gaussian PTR model]\label{ex:scalarflat}
For the affine proper PTR submanifold of Example~\ref{ex:flatPTR}, the ambient curvature and second fundamental form vanish:
\[
c=0,\qquad h=0,\qquad H=0,\qquad \mathcal Q_h=0.
\]
The Gaussian potential is $f(x)=\frac{\lambda}{2}\|x-x_0\|^2$, so
\[
\Delta f=n\lambda.
\]
Hence both sides of \eqref{eq:scalarPTR} are exactly $n\lambda$.  This verifies the scalar identity on an actual nonconstant gradient Ricci soliton and also shows why the PTR tangency energy disappears from the trace formula in the flat case.
\end{example}

\begin{example}[Scalar check on the Segre model]\label{ex:SegreScalar}
Continue Example~\ref{ex:SegrePTR}.  For the standard Segre $CR$-product, the pure factor components of the second fundamental form vanish; see Chen \cite{Chen1981b}. Hence
the pure components of \(h\) vanish and only mixed components remain.  For
orthonormal frames
\[
\{X_1,X_2\}\subset\mathscr D,
\qquad
\{Z_1,\ldots,Z_5\}\subset\mathscr D^\perp,
\]
consider a mixed plane \(\operatorname{span}\{X_i,Z_a\}\). Its sectional
curvature in the product metric is zero, because in a Riemannian product a
$2$-plane spanned by vectors tangent to two different factors has zero mixed
sectional curvature. Since \(JZ_a\) is normal, the ambient
complex-space-form contribution is \(c/4=1\). The Gauss equation gives
\[
0=1-\|h(X_i,Z_a)\|^2,
\]
so
\[
\|h(X_i,Z_a)\|=1.
\]
Thus \(H=0\), and with the ordered-pair Hilbert--Schmidt convention,
\[
\|h\|^2
=
2\sum_{i=1}^{2}\sum_{a=1}^{5}\|h(X_i,Z_a)\|^2
=20.
\]
Since \(c=4\), \(n=7\), \(\lambda=4\), \(f\) is constant, and
\(\|P\|^2=2\), the two sides of \eqref{eq:scalarPTR} are
\[
-20
\]
and
\[
7\cdot4-\frac{4\cdot7\cdot6}{4}
-\frac{3\cdot4}{4}\cdot2
=28-42-6=-20,
\]
respectively.
\end{example}

\begin{corollary}[Compact integral identity]\label{cor:compact}
If $N$ is compact without boundary, then
\begin{equation}\label{eq:compact}
 n\lambda\operatorname{Vol}(N)
=
\frac{cn(n-1)}4\operatorname{Vol}(N)
+\frac{3c}{4}\int_N\|P\|^2\,dV
+\int_N\bigl(n^2\|H\|^2-\|h\|^2\bigr)\,dV.
\end{equation}
\end{corollary}

\begin{proof}
Integrate the scalar identity over \(N\):
\begin{align*}
&\int_N\Delta f\,dV
+n^2\int_N\|H\|^2\,dV
-\int_N\|h\|^2\,dV\\
&\qquad=
n\lambda\operatorname{Vol}(N)
-\frac{cn(n-1)}4\operatorname{Vol}(N)
-\frac{3c}{4}\int_N\|P\|^2\,dV.
\end{align*}
Because \(N\) is compact and has no boundary,
\[
\int_N\Delta f\,dV
=
\int_N\operatorname{div}(\nabla f)\,dV
=0
\]
by the divergence theorem. Moving the remaining terms to the right-hand side gives exactly \eqref{eq:compact}.
\end{proof}

\begin{theorem}[Maximum-principle bound for compact minimal PTR solitons]\label{thm:maxbound}
Let $N^n$ be a compact minimal PTR-submanifold of a complex space form
$\widetilde M(c)$, $c>0$, carrying a gradient Ricci soliton. Then
\begin{equation}\label{eq:properbound}
\boxed{\lambda\le\frac{c(n-1)}4.}
\end{equation}
The bound follows from the following pointwise estimate. Put
\[
a=\lambda-\frac{c(n-1)}4.
\]
If $p$ is a maximum point of $f$ and $X_\alpha\in T_pN$ is a unit eigenvector
of $-P^2$ with eigenvalue $\rho_\alpha^2$, then
\begin{equation}\label{eq:maxpoint}
a\le \frac{3c}{4}\rho_\alpha^2
-\sum_{i=1}^n\|h(X_\alpha,e_i)\|^2
\le\frac{3c}{4}\rho_\alpha^2.
\end{equation}
Thus \eqref{eq:maxpoint} gives the directionwise refinement, including the
second-fundamental-form contribution, while \eqref{eq:properbound} is the
global PTR bound under the convention adopted after Definition~\ref{def:PTR}.
\end{theorem}

\begin{proof}
At a maximum point $p$ of $f$, $\operatorname{Hess}f\le0$.  From
Lemma~\ref{lem:Ricci} and the soliton equation,
\[
S_f+\mathcal Q_h
=
\left(\lambda-\frac{c(n-1)}4\right)I+\frac{3c}{4}P^2.
\]
Since $H=0$,
\[
g(\mathcal Q_hX,X)=-\sum_i\|h(X,e_i)\|^2.
\]
For $P^2X_\alpha=-\rho_\alpha^2X_\alpha$,
\[
\operatorname{Hess}f(X_\alpha,X_\alpha)
=
a-\frac{3c}{4}\rho_\alpha^2
+\sum_i\|h(X_\alpha,e_i)\|^2\le0,
\]
which proves \eqref{eq:maxpoint}.

By the PTR convention used in this paper, $\mathscr H^\perp$ has positive
rank.  Since it is totally real, $\mathscr H^\perp_p\subset\ker P_p$ at every
point.  Choose a unit vector $X\in\mathscr H^\perp_p$.  Then $PX=0$, so
\eqref{eq:maxpoint} with $\rho^2=0$ gives
\[
\lambda-\frac{c(n-1)}4
\le -\sum_i\|h(X,e_i)\|^2\le0.
\]
This is exactly \eqref{eq:properbound}.
\end{proof}

\begin{remark}[Why no separate $\rho_{\min}$ bound is needed]\label{rem:rhomin}
If one works with a broader formal splitting that allows the totally real
block to vanish, the directionwise estimate \eqref{eq:maxpoint} would imply
\[
\lambda\le\frac{c(n-1)}4+\frac{3c}{4}\rho_{\min}^2(p).
\]
Under the PTR convention adopted here, however, $\mathscr H^\perp_p\ne0$ and
$\mathscr H^\perp_p\subset\ker P_p$ at every point, so
$\rho_{\min}^2(p)=0$.  Hence the displayed inequality reduces identically to
\eqref{eq:properbound}; it is therefore recorded only as a limiting
comparison rather than as a separate conclusion of the theorem.
\end{remark}

\begin{remark}[The integrated estimate is weaker]\label{rem:integratedweak}
The compact integral identity alone gives
\[
\lambda\le\frac{c(n-1)}4+\frac{3cq}{4n},
\qquad q=\dim\widetilde{\mathscr H},
\]
using \(\|P\|^2\le q\) and \(\|h\|^2\ge0\).
Under the convention adopted here, $\mathscr H^\perp$ already has positive
rank. Thus the only additional condition for the mixed case is
$\widetilde{\mathscr H}\ne0$, equivalently $q>0$. In that case
Theorem~\ref{thm:maxbound} is strictly stronger than the integrated estimate,
because
\[
\lambda\le\frac{c(n-1)}4
<\frac{c(n-1)}4+\frac{3cq}{4n}.
\]
Consequently, equality in the weaker integrated estimate is impossible whenever
$q>0$.  We therefore retain the integral
identity but do not treat its weaker bound as a principal result.
\end{remark}

\begin{corollary}[Strictness from the totally real block]\label{cor:strictproper}
Assume the hypotheses of Theorem~\ref{thm:maxbound}. Let
\[
\mathcal N_p(h)
=
\{X\in T_pN:\ h(X,Y)=0\ \text{for every }Y\in T_pN\}
\]
be the relative nullity space. If, at some maximum point $p$ of $f$,
\[
\mathscr H^\perp_p\cap\mathcal N_p(h)=\{0\},
\]
then
\[
\lambda<\frac{c(n-1)}4.
\]
In particular, the conclusion holds when the displayed intersection is zero
at every point of $N$.
\end{corollary}

\begin{proof}
Choose a maximum point \(p\) of \(f\) at which the stated relative-nullity
condition holds, and take a unit vector
\(X\in\mathscr H^\perp_p\). Then \(PX=0\). The hypothesis implies
\(X\notin\mathcal N_p(h)\), hence
\[
\sum_i\|h(X,e_i)\|^2>0.
\]
The first inequality in \eqref{eq:maxpoint}, with \(\rho^2=0\), gives
\[
\lambda-\frac{c(n-1)}4
\le
-\sum_i\|h(X,e_i)\|^2<0.
\]
\end{proof}

\begin{example}[Strictness on the Segre model]\label{ex:SegreStrict}
In Example~\ref{ex:SegrePTR}, every unit
\(Z\in\mathscr H^\perp=T\mathbb{RP}^{5}\) satisfies
\[
\sum_i\|h(Z,e_i)\|^2=2.
\]
Thus Corollary~\ref{cor:strictproper} applies.  Indeed,
\[
\lambda=4
<
\frac{4(7-1)}4
=6.
\]
This gives a nonflat mixed PTR model (of CR type, hence non-proper) for which the effective bound is
strict.
\end{example}

\begin{remark}[Standard endpoint comparison]\label{ex:maxequality}
For the totally geodesic real form
$\mathbb{RP}^{n}(c/4)\subset\mathbb{CP}^{n}(c)$, which is a totally real
PTR submanifold, one has $P=0$, $h=0$, and
\[
\lambda=\frac{c(n-1)}4.
\]
Thus the zero-level PTR bound in Theorem~\ref{thm:maxbound} is attained.
For comparison only, the totally geodesic holomorphic inclusion
$\mathbb{CP}^{r}(c)\subset\mathbb{CP}^{m}(c)$ satisfies $P^2=-I$, $h=0$,
and
\[
\lambda=\frac{c(n+2)}4.
\]
This holomorphic model verifies the opposite algebraic endpoint of the
operator formula but is not being asserted to be a PTR example in the sense
of Definition~2.1.
\end{remark}

\begin{theorem}[A compact triviality criterion]\label{thm:constantPhi}
Let $N^n$ be a compact connected PTR-submanifold of $\widetilde M(c)$ carrying
a gradient Ricci soliton. Define
\[
\Phi=\frac{3c}{4}\|P\|^2+n^2\|H\|^2-\|h\|^2.
\]
If $\Phi$ is constant, then $f$ is constant and $N$ is Einstein:
\[
\operatorname{Ric}=\lambda g,\qquad
\lambda=\frac{c(n-1)}4+\frac{\Phi}{n}.
\]
\end{theorem}

\begin{proof}
The scalar-curvature calculation in Proposition~\ref{prop:scalar} gives
\[
\operatorname{Scal}_N=\frac{cn(n-1)}4+\Phi,
\]
so $\operatorname{Scal}_N$ is constant. Tracing the soliton equation yields
\[
\Delta f=n\lambda-\operatorname{Scal}_N,
\]
which is constant because $\lambda\in\mathbb R$. Compactness gives
\[
0=\int_N\Delta f\,dV=(n\lambda-\operatorname{Scal}_N)\operatorname{Vol}(N),
\]
hence $\Delta f=0$. Since $N$ is compact and connected, $f$ is constant.
The soliton equation reduces to $\operatorname{Ric}=\lambda g$, and
$n\lambda=\operatorname{Scal}_N$ gives the stated formula.
\end{proof}

\begin{remark}
The criterion applies, in particular, when $\|P\|^2$, $\|H\|^2$, and
$\|h\|^2$ are all constant. Thus many high-symmetry compact candidates can
support only trivial gradient solitons. It does not exclude noncompact
examples or compact examples for which $\Phi$ varies.
\end{remark}

\begin{example}[The Segre model and the triviality criterion]\label{ex:SegrePhi}
For Example~\ref{ex:SegrePTR},
\[
\Phi
=
\frac{3\cdot4}{4}\cdot2-\|h\|^2
=
6-20=-14.
\]
Thus \(\Phi\) is constant, and Theorem~\ref{thm:constantPhi} predicts that the
compact soliton must be trivial.  Moreover,
\[
\lambda
=
\frac{4(7-1)}4+\frac{-14}{7}
=
6-2=4,
\]
exactly as in the direct product calculation.  This is an internal
consistency check; it does not provide a nonflat example with nonconstant
potential.
\end{example}

\section{Reconstruction, spectral separation, and anisotropy}\label{sec:spectral}

We now use Lemma~\ref{lem:Ricci} together with the gradient soliton equation.

\begin{theorem}[Soliton--extrinsic reconstruction]\label{thm:reconstruction}
Let $N^n$ be a PTR-submanifold of a complex space form $\widetilde M(c)$ carrying a gradient Ricci soliton. Then
\begin{equation}\label{eq:operator}
\boxed{
S_f+\mathcal Q_h
=
\left(\lambda-\frac{c(n-1)}4\right)I
+\frac{3c}{4}P^2.
}
\end{equation}
If $c\ne0$, then
\begin{equation}\label{eq:reconstruct}
P^2
=
\frac4{3c}
\left[
S_f+\mathcal Q_h
-\left(\lambda-\frac{c(n-1)}4\right)I
\right].
\end{equation}
\end{theorem}

\begin{proof}
Lemma~\ref{lem:Ricci} gives, for all tangent vectors \(X,Y\),
\[
\operatorname{Ric}(X,Y)
=
\frac{c(n-1)}4g(X,Y)
-\frac{3c}{4}g(P^2X,Y)
+g(\mathcal Q_hX,Y).
\]
Therefore the Ricci endomorphism is
\[
\operatorname{Ric}^{\sharp}
=
\frac{c(n-1)}4I-\frac{3c}{4}P^2+\mathcal Q_h.
\]
The soliton equation is equivalent to
\[
\operatorname{Ric}^{\sharp}+S_f=\lambda I.
\]
Substitution gives
\[
\frac{c(n-1)}4I-\frac{3c}{4}P^2+\mathcal Q_h+S_f
=
\lambda I.
\]
Move the scalar term and the \(P^2\)-term to the opposite sides:
\[
S_f+\mathcal Q_h
=
\left(\lambda-\frac{c(n-1)}4\right)I
+\frac{3c}{4}P^2.
\]
This proves \eqref{eq:operator}.

If \(c\ne0\), subtract the scalar multiple of the identity:
\[
S_f+\mathcal Q_h
-\left(\lambda-\frac{c(n-1)}4\right)I
=
\frac{3c}{4}P^2.
\]
Multiplication by \(4/(3c)\) yields \eqref{eq:reconstruct}.
\end{proof}

\begin{example}[Reconstruction on the Segre model]\label{ex:SegreReconstruction}
Continue Example~\ref{ex:SegrePTR}.  Because \(H=0\), the mixed geometry gives
\[
\mathcal Q_hX=-5X
\quad(X\in\mathscr D),
\qquad
\mathcal Q_hZ=-2Z
\quad(Z\in\mathscr D^\perp).
\]
Also \(S_f=0\) and
\[
\lambda-\frac{c(n-1)}4
=
4-\frac{4\cdot6}{4}
=-2.
\]
Hence \eqref{eq:reconstruct} yields
\[
P^2X
=
\frac13(-5X+2X)
=-X
\]
and
\[
P^2Z
=
\frac13(-2Z+2Z)
=0.
\]
The reconstructed operator therefore agrees with the geometric decomposition
of Example~\ref{ex:SegrePTR}.
\end{example}

\begin{corollary}[Constraint on the totally real distribution]\label{cor:Hperpconstraint}
Let $a=\lambda-c(n-1)/4$. Then
\[
(S_f+\mathcal Q_h)X=aX
\]
for every $X\in\mathscr H^\perp$. Thus the totally real distribution is
contained pointwise in the eigenspace of $S_f+\mathcal Q_h$ corresponding to
the constant eigenvalue $a$.
\end{corollary}

\begin{proof}
For $X\in\mathscr H^\perp$ one has $PX=0$. Equation~\eqref{eq:operator}
therefore gives the assertion.
\end{proof}

\begin{proposition}[Pointwise spectral correspondence]\label{prop:spectralcorr}
Assume $c\ne0$. For a nonzero vector $X\in T_pN$ and a number $0\le\rho^2\le1$,
\[
P^2X=-\rho^2X
\]
if and only if
\begin{equation}\label{eq:eigcorr}
(S_f+\mathcal Q_h)X
=
\left[\lambda-\frac c4(n-1+3\rho^2)\right]X.
\end{equation}
Thus $P^2$ and $S_f+\mathcal Q_h$ have the same pointwise eigenspaces.
\end{proposition}

\begin{proof}
If $P^2X=-\rho^2X$, then \eqref{eq:operator} gives
\begin{align*}
(S_f+\mathcal Q_h)X
&=\left(\lambda-\frac{c(n-1)}4\right)X+\frac{3c}{4}P^2X\\
&=\left(\lambda-\frac{c(n-1)}4-\frac{3c}{4}\rho^2\right)X\\
&=\left[\lambda-\frac c4(n-1+3\rho^2)\right]X.
\end{align*}
Conversely, assume \eqref{eq:eigcorr}. Using \eqref{eq:reconstruct},
\begin{align*}
P^2X
&=\frac4{3c}\left[
(S_f+\mathcal Q_h)X-
\left(\lambda-\frac{c(n-1)}4\right)X
\right]\\
&=\frac4{3c}\left(-\frac{3c}{4}\rho^2X\right)
=-\rho^2X.
\end{align*}
\end{proof}

\begin{corollary}[Spectral interval and endpoint detection]\label{cor:interval}
Assume $q=\dim\widetilde{\mathscr H}>0$ and use the maximal constant-rank
choice of Remark~\ref{rem:canonicalPTR}. If $c>0$, every eigenvalue $\sigma$
of $(S_f+\mathcal Q_h)|_{\widetilde{\mathscr H}}$ satisfies
\begin{equation}\label{eq:interval}
\lambda-\frac{c(n+2)}4
\le \sigma<
\lambda-\frac{c(n-1)}4.
\end{equation}
The lower endpoint occurs precisely on invariant directions
$\|PX\|=\|X\|$. The excluded upper endpoint corresponds to the zero level
$PX=0$, whose eigenspace in the full tangent bundle is
$\mathscr H^\perp=\ker P$. For $c<0$, the order reverses:
\[
\lambda-\frac{c(n-1)}4
<\sigma\le
\lambda-\frac{c(n+2)}4.
\]
\end{corollary}

\begin{proof}
Let $X\in\widetilde{\mathscr H}$ be a unit $P^2$-eigenvector,
\[
P^2X=-\rho^2X.
\]
By Remark~\ref{rem:canonicalPTR},
$\widetilde{\mathscr H}=\operatorname{Im}P=(\ker P)^\perp$, so a nonzero
vector in $\widetilde{\mathscr H}$ cannot lie in $\ker P$. Lemma~\ref{lem:P2basic}
therefore gives
\[
0<\rho^2\le1.
\]
By Proposition~\ref{prop:spectralcorr},
\[
\sigma(\rho^2)=\lambda-\frac c4(n-1+3\rho^2).
\]
For $c>0$ this is decreasing in $\rho^2$, which gives
\[
\lambda-\frac{c(n+2)}4\le\sigma<\lambda-\frac{c(n-1)}4.
\]
The left equality holds exactly when $\rho^2=1$, equivalently
$\|PX\|=\|X\|$; then $FX=0$ and the direction is invariant. The missing
upper value corresponds to $\rho=0$, hence to $\ker P=\mathscr H^\perp$ in
the full tangent bundle. For $c<0$ the monotonicity reverses.
\end{proof}

\begin{remark}[Segre endpoint realization and standard comparison]\label{ex:standardextreme}
The Segre $CR$-product of Example~\ref{ex:SegrePTR} is a nonflat PTR
submanifold with $c=4$, $n=7$, $\lambda=4$, and
$\widetilde{\mathscr H}=\mathscr D$. By
Example~\ref{ex:SegreReconstruction},
\[
(S_f+\mathcal Q_h)|_{\mathscr D}=-5I,
\qquad
(S_f+\mathcal Q_h)|_{\mathscr D^\perp}=-2I.
\]
The lower endpoint in Corollary~\ref{cor:interval} is
\[
\lambda-\frac{c(n+2)}4
=4-\frac{4\cdot9}{4}=-5,
\]
so it is attained on $\mathscr D$ because $P^2|_{\mathscr D}=-I$. On the
totally real block,
\[
a=\lambda-\frac{c(n-1)}4=4-6=-2,
\]
which verifies Corollary~\ref{cor:Hperpconstraint} on
$\mathscr D^\perp=T\mathbb{RP}^5$. Thus one nonflat PTR model exhibits
both the invariant endpoint on the ambiguous distribution and the
zero-tangency eigenvalue on the complementary totally real block.

For the totally real projective form
$\mathbb{RP}^{n}(c/4)\subset\mathbb{CP}^{n}(c)$ one has $P^2=0$,
$\mathcal Q_h=0$, $S_f=0$, and $\lambda=c(n-1)/4$; it realizes the
zero-level endpoint globally, although $q=0$ so the restricted interval on
$\widetilde{\mathscr H}$ is not applicable. For comparison only, the
holomorphic inclusion $\mathbb{CP}^{r}(c)\subset\mathbb{CP}^{m}(c)$ has
$P^2=-I$, $\mathcal Q_h=0$, $S_f=0$, and $\lambda=c(n+2)/4$. In accordance
with the PTR convention fixed after Definition~\ref{def:PTR}, this last
model is only an external K\"ahler comparison and is not used as a PTR
example.
\end{remark}

\begin{proposition}[Spectral separation]\label{prop:separation}
Let $c\ne0$, and suppose $X_i,X_j$ are unit eigenvectors of $P^2$ with squared tangency levels $\rho_i^2$ and $\rho_j^2$. Let $\sigma_i,\sigma_j$ be the corresponding eigenvalues of $S_f+\mathcal Q_h$. Then
\begin{equation}\label{eq:separation}
\sigma_i-\sigma_j
=-\frac{3c}{4}\bigl(\rho_i^2-\rho_j^2\bigr).
\end{equation}
In particular, distinct tangency levels are separated by the soliton--extrinsic spectrum whenever $c\ne0$.
\end{proposition}

\begin{proof}
For \(X_i\) and \(X_j\), Proposition~\ref{prop:spectralcorr} gives
\[
\sigma_i
=
\lambda-\frac c4(n-1+3\rho_i^2)
\]
and
\[
\sigma_j
=
\lambda-\frac c4(n-1+3\rho_j^2).
\]
Subtracting,
\begin{align*}
\sigma_i-\sigma_j
&=
-\frac{3c}{4}\rho_i^2+\frac{3c}{4}\rho_j^2\\
&=
-\frac{3c}{4}(\rho_i^2-\rho_j^2).
\end{align*}
If \(c\ne0\) and \(\rho_i^2\ne\rho_j^2\), the right-hand side is nonzero, so \(\sigma_i\ne\sigma_j\).
\end{proof}

\begin{remark}[Algebraic spectral check]
If \(c=4\), \(\rho_i^2=1\), and \(\rho_j^2=1/4\), then
\[
\sigma_i-\sigma_j
=
-\frac{3\cdot4}{4}\left(1-\frac14\right)
=
-\frac94.
\]
This checks Proposition~\ref{prop:separation}; no nonflat realization is asserted.
\end{remark}

We now remove the common scalar part of \eqref{eq:operator}. Set
\[
q=\dim\widetilde{\mathscr H}.
\]
Assume $q>0$. For a self-adjoint endomorphism $A$ of $\widetilde{\mathscr H}$, write
\[
A^{\circ}=A-\frac{\operatorname{tr}A}{q}I.
\]

\begin{theorem}[Exact anisotropy identity]\label{thm:anisotropy}
Assume $c\ne0$ and $q=\dim\widetilde{\mathscr H}>0$. On the ambiguous distribution,
\begin{equation}\label{eq:tracefree}
\left[(S_f+\mathcal Q_h)|_{\widetilde{\mathscr H}}\right]^{\circ}
=
\frac{3c}{4}
\left[P^2|_{\widetilde{\mathscr H}}\right]^{\circ}.
\end{equation}
Let the eigenvalues of $P^2|_{\widetilde{\mathscr H}}$ be
$-\rho_1^2,\dots,-\rho_q^2$, repeated according to multiplicity, and put
\[
\overline{\rho^2}
=\frac1q\sum_{\alpha=1}^{q}\rho_\alpha^2.
\]
Then
\begin{equation}\label{eq:anisotropy}
\boxed{
\left\|
\left[(S_f+\mathcal Q_h)|_{\widetilde{\mathscr H}}\right]^{\circ}
\right\|_{\mathrm{HS}}^2
=
\frac{9c^2}{16}
\sum_{\alpha=1}^{q}
\left(\rho_\alpha^2-\overline{\rho^2}\right)^2.
}
\end{equation}
\end{theorem}

\begin{proof}
Set
\[
A=(S_f+\mathcal Q_h)|_{\widetilde{\mathscr H}},
\qquad
B=P^2|_{\widetilde{\mathscr H}},
\qquad
a=\lambda-\frac{c(n-1)}4.
\]
Equation \eqref{eq:operator} becomes
\[
A=aI+\frac{3c}{4}B.
\]
Taking the trace on the \(q\)-dimensional space \(\widetilde{\mathscr H}\),
\[
\operatorname{tr}A
=
aq+\frac{3c}{4}\operatorname{tr}B.
\]
Therefore
\begin{align*}
A^\circ
&=
A-\frac{\operatorname{tr}A}{q}I\\
&=
aI+\frac{3c}{4}B
-\left(
a+\frac{3c}{4q}\operatorname{tr}B
\right)I\\
&=
\frac{3c}{4}
\left(
B-\frac{\operatorname{tr}B}{q}I
\right)\\
&=
\frac{3c}{4}B^\circ.
\end{align*}
This proves \eqref{eq:tracefree}.

Now choose an orthonormal eigenbasis
\[
BE_\alpha=-\rho_\alpha^2E_\alpha.
\]
Then
\[
\operatorname{tr}B
=
-\sum_{\alpha=1}^{q}\rho_\alpha^2
=
-q\,\overline{\rho^2}.
\]
Hence
\begin{align*}
B^\circ E_\alpha
&=
\left(
B-\frac{\operatorname{tr}B}{q}I
\right)E_\alpha\\
&=
-\rho_\alpha^2E_\alpha
+\overline{\rho^2}\,E_\alpha\\
&=
-\left(\rho_\alpha^2-\overline{\rho^2}\right)E_\alpha.
\end{align*}
Since \(B^\circ\) is self-adjoint,
\[
\|B^\circ\|_{\mathrm{HS}}^2
=
\sum_{\alpha=1}^{q}
\left(\rho_\alpha^2-\overline{\rho^2}\right)^2.
\]
Finally,
\[
A^\circ=\frac{3c}{4}B^\circ
\]
implies
\[
\|A^\circ\|_{\mathrm{HS}}^2
=
\frac{9c^2}{16}\|B^\circ\|_{\mathrm{HS}}^2,
\]
which is \eqref{eq:anisotropy}.
\end{proof}

\begin{example}[Concrete check on the Segre model]\label{ex:SegreAnisotropy}
On $\widetilde{\mathscr H}=\mathscr D$ in Example~\ref{ex:SegrePTR},
$q=2$ and $P^2|_{\mathscr D}=-I$. Hence
\[
[P^2|_{\mathscr D}]^\circ=0.
\]
At the same time Example~\ref{ex:SegreReconstruction} gives
$(S_f+\mathcal Q_h)|_{\mathscr D}=-5I$, so
\[
[(S_f+\mathcal Q_h)|_{\mathscr D}]^\circ=0.
\]
Thus both sides of \eqref{eq:tracefree} vanish. Moreover
$\rho_1^2=\rho_2^2=1$, so the right-hand side of
\eqref{eq:anisotropy} is also zero. The nonflat Segre model therefore
verifies both forms of Theorem~\ref{thm:anisotropy} directly.

\end{example}

\begin{corollary}[Single-level detection]\label{cor:single}
Assume $c\ne0$ and $q=\dim\widetilde{\mathscr H}>0$. On $\widetilde{\mathscr H}$ the following are equivalent:
\begin{enumerate}[label=\textnormal{(\roman*)}]
\item $(S_f+\mathcal Q_h)|_{\widetilde{\mathscr H}}$ is a scalar multiple of the identity;
\item $[P^2|_{\widetilde{\mathscr H}}]^\circ=0$;
\item $P^2|_{\widetilde{\mathscr H}}=-\rho^2I$ for a pointwise function $\rho$.
\end{enumerate}
Hence the soliton--extrinsic operator detects exactly when the ambiguous distribution has a single pointwise tangency level.
\end{corollary}

\begin{proof}
If (i) holds, its trace-free part vanishes. Equation~\eqref{eq:tracefree} and $c\ne0$ imply (ii). Conversely, (ii) and \eqref{eq:tracefree} imply that the trace-free part of $(S_f+\mathcal Q_h)|_{\widetilde{\mathscr H}}$ vanishes, which is equivalent to (i).

Condition (ii) means
\[
P^2|_{\widetilde{\mathscr H}}
=
\frac{\operatorname{tr}(P^2|_{\widetilde{\mathscr H}})}q I.
\]
Since $P^2$ is nonpositive, the scalar can be written $-\rho^2$, proving (iii). Conversely, if $P^2|_{\widetilde{\mathscr H}}=-\rho^2I$, then it is scalar and hence its trace-free part vanishes.
\end{proof}

\begin{remark}[Single-level detection on the Segre model]
For Example~\ref{ex:SegrePTR}, on
$\widetilde{\mathscr H}=\mathscr D$ one has simultaneously
\[
P^2|_{\mathscr D}=-I,
\qquad
[P^2|_{\mathscr D}]^\circ=0,
\qquad
(S_f+\mathcal Q_h)|_{\mathscr D}=-5I.
\]
Thus conditions (i), (ii), and (iii) of Corollary~\ref{cor:single} all hold,
with the single tangency level $\rho^2=1$. A strictly intermediate
pointwise slant ambiguous block would instead have $0<\rho^2<1$. The
totally real projective form has $q=0$, so the trace-free criterion is not
applicable to it.
\end{remark}

\begin{corollary}[Multi-level anisotropy]\label{cor:multilevel}
If $c\ne0$ and $P^2|_{\widetilde{\mathscr H}}$ has at least two distinct eigenvalues, then
\[
S_f|_{\widetilde{\mathscr H}}
\quad\text{and}\quad
\mathcal Q_h|_{\widetilde{\mathscr H}}
\]
cannot both be scalar multiples of the identity.
\end{corollary}

\begin{proof}
Suppose, to the contrary, that
\[
S_f|_{\widetilde{\mathscr H}}=\mu I,
\qquad
\mathcal Q_h|_{\widetilde{\mathscr H}}=\kappa I.
\]
Then
\[
(S_f+\mathcal Q_h)|_{\widetilde{\mathscr H}}
=
(\mu+\kappa)I.
\]
By Corollary~\ref{cor:single},
\[
P^2|_{\widetilde{\mathscr H}}=-\rho^2I.
\]
Thus \(P^2|_{\widetilde{\mathscr H}}\) has only one eigenvalue, contradicting the hypothesis that at least two distinct eigenvalues occur.
\end{proof}

\begin{example}[Algebraic multi-level illustration]\label{ex:algebraicMultilevel}
To respect the even-multiplicity property from Lemma~\ref{lem:P2basic}, take
$q=4$, $c=4$, and suppose at one point that the squared tangency levels are
\[
\rho_1^2=\rho_2^2=1,
\qquad
\rho_3^2=\rho_4^2=\frac14.
\]
Then $\overline{\rho^2}=5/8$ and
\[
\sum_{\alpha=1}^{4}
(\rho_\alpha^2-\overline{\rho^2})^2
=4\left(\frac38\right)^2
=\frac9{16}.
\]
Theorem~\ref{thm:anisotropy} therefore gives
\[
\left\|
[(S_f+\mathcal Q_h)|_{\widetilde{\mathscr H}}]^\circ
\right\|_{\mathrm{HS}}^2
=9\cdot\frac9{16}=\frac{81}{16}>0.
\]
Consequently $S_f|_{\widetilde{\mathscr H}}$ and
$\mathcal Q_h|_{\widetilde{\mathscr H}}$ cannot both be scalar, exactly
as Corollary~\ref{cor:multilevel} asserts. This is an algebraic
illustration only; no nonflat geometric realization with these two levels is
being asserted.
\end{example}

For any tangent-bundle endomorphism $A$, we use the standard induced
covariant derivative
\[
(\nabla_XA)Y:=\nabla_X(AY)-A(\nabla_XY).
\]
In particular, this convention applies to $P^2$, $S_f$, and $\mathcal Q_h$.

\begin{proposition}[Differential soliton constraint]\label{prop:differential}
For every gradient Ricci soliton considered above,
\begin{align}\label{eq:differential}
R(X,Y)\nabla f
={}&
\frac{3c}{4}
\bigl((\nabla_XP^2)Y-(\nabla_YP^2)X\bigr) \notag\\
&-
\bigl((\nabla_X\mathcal Q_h)Y-(\nabla_Y\mathcal Q_h)X\bigr).
\end{align}
\end{proposition}

\begin{proof}
For the Hessian endomorphism \(S_fX=\nabla_X\nabla f\), the Ricci identity gives
\[
R(X,Y)\nabla f
=
(\nabla_XS_f)Y-(\nabla_YS_f)X.
\]
From \eqref{eq:operator},
\[
S_f
=
aI+\frac{3c}{4}P^2-\mathcal Q_h,
\qquad
a=\lambda-\frac{c(n-1)}4.
\]
The number \(a\) is constant, so differentiating this identity and
antisymmetrizing in \(X,Y\) gives \eqref{eq:differential}.
\end{proof}

\begin{corollary}[Parallel data and compact triviality]\label{cor:nullity}
If
\[
\nabla P^2=0,
\qquad
\nabla\mathcal Q_h=0,
\]
then
\[
R(X,Y)\nabla f=0
\]
for all \(X,Y\), so \(\nabla f\) lies in the curvature-nullity
distribution. If the curvature-nullity distribution is zero at every point,
then $\nabla f=0$ and hence $f$ is constant and the soliton is Einstein,
without any compactness assumption. If instead $N$ is compact and connected,
then the same conclusion follows without assuming zero curvature nullity. In
particular, on a compact connected
$N$, the stronger geometric assumptions
\[
\nabla P^2=0,
\qquad
\nabla h=0
\]
also force the soliton to be trivial.
\end{corollary}

\begin{proof}
The curvature-nullity assertion follows directly from
Proposition~\ref{prop:differential}. If the curvature-nullity distribution is
zero at every point, then the inclusion $\nabla f$ in that distribution forces
$\nabla f=0$, so $f$ is constant and the soliton equation reduces to
$\operatorname{Ric}=\lambda g$; no compactness is needed for this conclusion.
For the compact alternative, contract $R(X,Y)\nabla f=0$ in the first and
fourth slots with an orthonormal frame
$\{e_i\}$:
\[
\operatorname{Ric}(Y,\nabla f)
=\sum_i g(R(e_i,Y)\nabla f,e_i)=0.
\]
For a gradient Ricci soliton satisfying
$\operatorname{Ric}+\operatorname{Hess}f=\lambda g$, the contracted Bianchi
identity gives the standard scalar-curvature identity
\[
\frac12\,\nabla\operatorname{Scal}_N
=\operatorname{Ric}(\nabla f,\cdot).
\]
Hence $\operatorname{Scal}_N$ is constant. Taking the trace of the soliton
equation,
\[
\operatorname{Scal}_N+\Delta f=n\lambda,
\]
shows that $\Delta f$ is constant as well. If $N$ is compact, integration
gives
\[
0=\int_N\Delta f\,dV
=(\Delta f)\operatorname{Vol}(N),
\]
so $\Delta f=0$. Compactness and connectedness then imply that the harmonic
function $f$ is constant. The soliton equation reduces to
$\operatorname{Ric}=\lambda g$.

Finally, $\nabla h=0$ implies that $H$ and all contractions entering
$\mathcal Q_h$ are parallel, hence $\nabla\mathcal Q_h=0$. This proves the
last assertion.
\end{proof}

\section{Geometric consequences, comparison, and scope}

The main compact consequences are now constraint and triviality results rather
than consequences of the trace estimate alone.  Theorem~\ref{thm:maxbound}
uses the maximum principle to control the soliton constant pointwise, while
Theorem~\ref{thm:constantPhi} shows that constancy of a natural scalar
combination of tangency and extrinsic data forces the soliton to be Einstein.
The reconstruction and anisotropy results of Section~\ref{sec:spectral} then
describe how the remaining nonflat, potentially nontrivial cases must encode
their PTR tangency spectrum.

The flat model of Example~\ref{ex:flatPTR} has nonconstant potential and
several nonzero tangency levels, but \(c=0\), so the nonflat reconstruction
mechanism disappears.  Example~\ref{ex:SegrePTR} is nonflat and
non-totally-geodesic, but its potential is constant; by
Example~\ref{ex:SegrePhi}, this triviality is exactly what
Theorem~\ref{thm:constantPhi} predicts. In a different noncompact and
homogeneous direction, Cidre-D\'iaz--Sanmart\'in-L\'opez
\cite{CidreDiazSanmartinLopez2026} classify Ricci-soliton Lie subgroups in
complex hyperbolic spaces. Their examples are not PTR gradient-soliton
models of the type sought here, so they serve as a nearby comparison rather
than as existence examples for the present theory.

\begin{proposition}[Nonexistence on compact real hypersurfaces]\label{prop:realhyp}
Let \(m\ge2\), and let \(M^{2m-1}\) be a compact real hypersurface of a
nonflat complex space form.  Denote its shape operator by \(A\).  If
\begin{equation}\label{eq:shapecombination}
(\operatorname{tr}A)^2-\operatorname{tr}(A^2)
\end{equation}
is constant, then \(M\) admits no gradient Ricci soliton.  In particular,
compact real hypersurfaces with constant principal curvatures, including the
standard homogeneous families, admit no gradient Ricci soliton.
\end{proposition}

\begin{proof}
For a real hypersurface with unit normal \(\nu\), put \(\xi=-J\nu\) and define \(\eta(X)=g(X,\xi)\). Then
\[
P^2=-I+\eta\otimes\xi,
\qquad
\|P\|^2=2m-2.
\]
If \(n=2m-1\), then
\[
H=\frac{\operatorname{tr}A}{n}\nu,
\qquad
\|h\|^2=\operatorname{tr}(A^2).
\]
Hence the quantity in Theorem~\ref{thm:constantPhi} is
\[
\Phi
=
\frac{3c}{4}(2m-2)
+
(\operatorname{tr}A)^2
-
\operatorname{tr}(A^2).
\]
Under the hypothesis, \(\Phi\) is constant.  If a gradient Ricci soliton
existed, Theorem~\ref{thm:constantPhi} would force \(f\) to be constant and
\(M\) to be Einstein.  However, real hypersurfaces of nonflat complex space
forms are not Einstein for \(m\ge2\); see
Cecil--Ryan \cite[Theorem~8.69]{CecilRyan2015}.  This obstruction is also
used in the Ricci-soliton setting by Cho--Kimura \cite{ChoKimura2009}.
For the classical projective and hyperbolic results, see also Cecil--Ryan
\cite{CecilRyan1982}, Montiel \cite{Montiel1985}, and Kon \cite{Kon1979}.  This contradiction proves
the nonexistence.
\end{proof}

\begin{corollary}[Necessary variation for compact real hypersurfaces]\label{cor:realhypvariation}
If a compact real hypersurface \(M^{2m-1}\), \(m\ge2\), of a nonflat complex
space form carries a gradient Ricci soliton, then
\[
(\operatorname{tr}A)^2-\operatorname{tr}(A^2)
\]
must be nonconstant.
\end{corollary}

\begin{proof}
This is the contrapositive of Proposition~\ref{prop:realhyp}.
\end{proof}

\begin{remark}
For a real hypersurface,
\[
\mathcal Q_h=(\operatorname{tr}A)A-A^2,
\]
and \eqref{eq:operator} becomes
\[
S_f=
\left(\lambda-\frac{c(n-1)}4\right)I
-\frac{3c}{4}(I-\eta\otimes\xi)
-(\operatorname{tr}A)A+A^2.
\]
Thus the general PTR constraint becomes an explicit shape-operator equation.
\end{remark}

Proposition~\ref{prop:differential} supplies a differential consequence of the
gradient structure that is not contained in the pointwise spectral
reconstruction alone.  It is consistent with the general rigidity phenomena
for gradient Ricci solitons discussed by Petersen--Wylie
\cite{PetersenWylie2009b}, while here the restrictions are expressed directly
through \(P^2\) and the extrinsic operator \(\mathcal Q_h\).

The results therefore identify both a constraint mechanism and a bridge from gradient-soliton data to PTR tangency data:
\[
(S_f,\mathcal Q_h)\quad\Longrightarrow\quad P^2
\quad\Longrightarrow\quad
\text{tangency spectrum and its anisotropy}.
\]
This bridge is effective only in nonzero ambient holomorphic sectional
curvature.  Example~\ref{ex:SegrePTR} shows that a nonflat,
non-totally-geodesic PTR embedding of CR type can carry a gradient Ricci soliton, but that
example is Einstein with constant potential and therefore does not resolve the
nontrivial existence problem.  At present we know no nonflat example in this
PTR setting with nonconstant potential, and in particular none whose
ambiguous distribution carries several distinct nonzero tangency levels.
Theorems~\ref{thm:maxbound} and \ref{thm:constantPhi}, together with
Proposition~\ref{prop:differential}, substantially restrict where such an
example could occur.

\end{document}